%% file: covincular.tex
\documentclass[12pt,reqno]{article}

\usepackage[usenames]{xcolor}
\usepackage{amssymb}
\usepackage{graphicx}
\usepackage{amscd}

\usepackage[colorlinks=true,
linkcolor=webgreen,
filecolor=webbrown,
citecolor=webgreen]{hyperref}

\definecolor{webgreen}{rgb}{0,.5,0}
\definecolor{webbrown}{rgb}{.6,0,0}

\usepackage{fullpage}
\usepackage{float}

\usepackage{graphics,amsmath,amssymb}
\usepackage{amsthm}
\usepackage{amsfonts}
\usepackage{latexsym}
\usepackage{epsf}

\newcommand{\seqnum}[1]{\href{http://oeis.org/#1}{\underline{#1}}}

\usepackage{xargs}
\usepackage{enumerate}
\usepackage{tikz}
\usetikzlibrary{patterns}
\usepackage{graphicx}
\usepackage{ifthen}
\usepackage{longtable}
\usepackage{fancyvrb}
\input{patternmacros}

\pgfmathsetmacro{\patttablescale}{1.05}
\pgfmathsetmacro{\pattdispscale}{0.80}
\pgfmathsetmacro{\patttextscale}{0.6}
\usepackage{subcaption}
\usepackage{authblk}

\usepackage[T1]{fontenc}

\newcommand{\sym}{\mathcal{S}}
\newcommand{\mcA}{\mathcal{A}}
\newcommand{\mcB}{\mathcal{B}}
\newcommand{\Av}{\mathrm{Av}}

\renewcommand{\P}{\mathcal{P}}
\newcommand{\Q}{\mathcal{R}}
\newcommand{\PQ}{\P\times\Q}

\def\addtolatticepath#1{%
  \expandafter\def\expandafter\latticepath\expandafter{\latticepath#1}%
}

\def\parselatticepath#1{%
  \def\latticepath{node {}}%
  \Parselatticepath#1@}
\def\Parselatticepath#1{%
  \ifx#1@%
    \let\next=\relax%
  \else%
    \csname latticepathletter#1\endcsname%
    \addtolatticepath{ node {} }%
    \let\next=\Parselatticepath
  \fi%
  \next}

\tikzset{%
  insert lattice path/.style={%
    every node/.style={
      circle,
      fill,
      draw=none,
      inner sep=1.6pt
    },
    insert path={%
      \pgfextra{\parselatticepath{#1}}%
      \latticepath
    }
  }
}

\begin{document}

\theoremstyle{plain}
\newtheorem{theorem}{Theorem}
\newtheorem{corollary}[theorem]{Corollary}
\newtheorem{lemma}[theorem]{Lemma}
\newtheorem{proposition}[theorem]{Proposition}
\newtheorem*{proposition*}{Proposition}

\theoremstyle{definition}
\newtheorem{definition}[theorem]{Definition}
\newtheorem{example}[theorem]{Example}
\newtheorem{conjecture}[theorem]{Conjecture}
\newtheorem{problem}[theorem]{Problem}

\theoremstyle{remark}
\newtheorem{remark}[theorem]{Remark}

\title{Enumerations of Permutations Simultaneously Avoiding a Vincular and a Covincular Pattern of Length $3$}
\author[1]{Christian Bean}
\author[2]{Anders Claesson}
\author[1]{Henning Ulfarsson}
\affil[1]{School of Computer Science, Reykjavik University, Menntavegi~1, 101~Reykjavik, Iceland}
\affil[2]{Science Institute, University of Iceland, Dunhaga~3, 107~Reykjavik, Iceland}

\maketitle

\begin{abstract}
    Vincular and covincular patterns are generalizations of classical patterns
    allowing restrictions on the indices and values of the occurrences in a
    permutation. In this paper we study the integer sequences arising as the
    enumerations of permutations simultaneously avoiding a vincular and a
    covincular pattern, both of length $3$, with at most one restriction. We see
    familiar sequences, such as the Catalan and Motzkin numbers, but also some
    previously unknown sequences which have close links to other combinatorial
    objects such as lattice paths and integer partitions. Where possible we
    include a generating function for the enumeration. One of the cases
    considered settles a conjecture by Pudwell (2010) on the Wilf-equivalence of
    barred patterns. We also give an alternative proof of the classic result
    that permutations avoiding $123$ are counted by the Catalan numbers.
\end{abstract}

\section{Introduction}
\thispagestyle{empty}

A permutation $\pi$ contains a \emph{classical pattern} $p$, which is itself a
permutation, if $\pi$ contains a \emph{subword} which is \emph{order isomorphic}
to $p$. Babson and Steingr\'{i}msson~\cite{vincular} introduced a generalization
of classical patterns that allows the requirement that two adjacent letters in a
pattern must be adjacent in the permutation. These are called \emph{vincular
patterns}. A further extension, called \emph{bivincular patterns}, was provided
by M. Bousquet-M\'{e}lou \emph{et al.}~\cite{bivincular}. We call the special
case when only constraints on values are allowed \emph{covincular patterns}. The
set of bivincular patterns is closed under the action of the symmetry group of
the square and an alternative way of describing the covincular patterns is that
they are inverses of vincular patterns.

Simultaneous avoidance of two vincular patterns was studied by Claesson and
Mansour \cite{vincularpair} and by Kitaev \cite{multiconsecutive}. Allowing one
of the patterns to be covincular is a natural follow up question and leads to
some well-known sequences.  The overall goal of this paper is to count the
number of permutations simultaneously avoiding a length $3$ vincular and a
length $3$ covincular pattern, where both patterns force at most one
restriction. A summary of our results can be found in Table~\ref{table:summary};
these results are detailed in Sections~\ref{classicsection} through
\ref{recurrences}. One of our methods can be adapted to give a new simple proof
of the classical result that permutations avoiding the classical pattern $123$
are counted by the Catalan numbers; see Section~\ref{avoiding123}. The Appendix
contains all the results from the paper collected by their respective
enumeration.

We now present the definitions and notation we use.  An \emph{alphabet}, $X$, is
a non-empty set. An element of $X$ is a $letter$. A finite sequence of letters
from $X$ is called a \emph{word}. The word with no letters is called the
\emph{empty word} and is denoted $\epsilon$. For a word $w$ we say that the
\emph{length} of the word, denoted $|w|$, is the number of letters in $w$, that
is, if $w = x_1x_2 \cdots x_n$ then $|w| = n$. A \emph{subword} of $w$ is a
finite sequence $x_{i_1}x_{i_2} \cdots x_{i_k}$ where $1 \leq i_1 < i_2 < \cdots <
i_k \leq n$.

As we are interested in permutations the alphabet we use is $[n] =
\{1,2,\ldots,n\}$ for some $n \in \mathbb{N} = \{0,1,2,\ldots\}$. A length $n$
permutation is a length $n$ word $x = x_1x_2 \cdots x_n$ of this alphabet with
no repeated letters where $x_i = \pi(i)$. Let $\sym_n$ denote the set of all
length $n$ permutations. Let $w = w_1w_2 \cdots w_k$ and $v = v_1v_2 \cdots v_k$
be words with distinct letters. We say that $w$ is \emph{order isomorphic} to
$v$ if, for all $i$ and $j$, we have $w_i < w_j$ precisely when $v_i < v_j$. For
example $53296$ and $32154$ are order isomorphic.

\begin{definition}[\mbox{Bousquet-M\'{e}lou \emph{et al.}~\cite[page 4]{bivincular}}]

  A \emph{bivincular pattern} is a triple, $p=(\sigma,X,Y)$, where
  $\sigma \in \sym_k$ is the \emph{underlying permutation}
  and $X$ and $Y$ are subsets of $\{0,1, \ldots, k\}$. An
  \emph{occurrence} of $p$ in $\pi\in\sym_n$ is a subsequence
  $w = \pi(i_1) \cdots \pi(i_k)$ order isomorphic to $\sigma$ such that
  \[
  \forall x\in X,\, i_{x+1} = i_x+1 \quad\text{and}\quad \forall y\in
  Y,\, j_{y+1} = j_y+1,
  \]
  where $\{\pi(i_1), \ldots, \pi(i_k)\}=\{j_1, \ldots, j_k\}$ and
  $j_1 < j_2 < \cdots < j_k$. By convention, $i_0=j_0=0$ and
  $i_{k+1}=j_{k+1}=n+1$. If such an occurrence exists we say that $\pi$
  \emph{contains} $\sigma$.
\end{definition}

We define the \emph{length} of a bivincular pattern $p=(\sigma,X,Y)$, denoted
$|p|$, to be $|\sigma|$.  Further, a permutation \emph{avoids} $p$ if it does
not contain $p$. If $Y = \emptyset$ then $p$ is a \emph{vincular} pattern. If $X =
\emptyset$ then $p$ is a \emph{covincular} pattern. If $X = Y = \emptyset$ then
$p$ is a \emph{classical} pattern.  For example, the permutation $15423$
contains an occurrence of $(123,\{2\},\emptyset)$, namely the subword $123$, but
avoids $(123,\{1\},\emptyset)$. The permutation $23514$ contains an occurrence
of $(312,\emptyset,\{2\})$, namely the subword $523$, but avoids
$(312,\emptyset,\{1\})$.  The sets of all length $n$ permutations avoiding the
pattern $p$ is denoted \[ \Av_n(p) = \left\{ \pi \in \sym_n : \pi \text{ avoids }
p \right\}, \] and, for $P$ a set of patterns, $\Av_n(P) = \cap_{p \in
P}\Av_n(p)$ and $\Av(P)=\cup_{n\geq 0}\Av_n(P)$.

Below we use a pictorial representation of vincular and covincular patterns. For
a length $n$ bivincular pattern $p=(\sigma,X,Y)$: First draw the collection of
points from the underlying permutation i.e.~ $(k, \sigma_k)$ where $1 \leq k
\leq n$. Then, for each $i \in X$, shade the $i$th column and, for each $j \in
Y$, shade the $j$th row; see Figure~\ref{fig:picture}. The shading is used to
denote the empty regions in the permutation if we were to overlay the grid onto
an occurrence of the pattern.

\begin{figure}[ht]
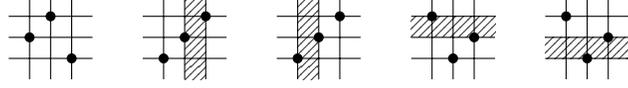

\begin{center}
\mpattern{scale=\pattdispscale}{ 3 }{ 1/2, 2/3, 3/1 }{}
\mpattern{scale=\pattdispscale}{ 3 }{ 1/1, 2/2, 3/3 }{ 2/0, 2/1, 2/2, 2/3 }
\mpattern{scale=\pattdispscale}{ 3 }{ 1/1, 2/2, 3/3 }{ 1/0, 1/1, 1/2, 1/3 }
\mpattern{scale=\pattdispscale}{ 3 }{ 1/3, 2/1, 3/2 }{ 0/2, 1/2, 2/2, 3/2}
\mpattern{scale=\pattdispscale}{ 3 }{ 1/3, 2/1, 3/2 }{ 0/1, 1/1, 2/1, 3/1}
\caption{$(231,\emptyset,\emptyset)$, $(123,\{2\},\emptyset)$, $(123,\{1\},\emptyset)$,
$(312,\emptyset,\{2\})$, and $(312,\emptyset,\{1\})$}
\label{fig:picture}
\end{center}
\end{figure}

\begin{remark}\label{subpattern}
  If $p = (\sigma,X,Y)$ and $p' = (\sigma,X',Y')$, where
  $X' \subseteq X$ and $Y' \subseteq Y$, then we immediately have that
  $\Av_n(p') \subseteq \Av_n(p)$.
\end{remark}

We are interested in $|\Av_n(p,r)|$ where $p=(\sigma,X,\emptyset)$ is a length
$3$ vincular pattern with $X\in \{\emptyset,\{1\},\{2\}\}$, and $r =
(\tau,\emptyset,Y)$ is a length $3$ covincular pattern with
$Y\in\{\emptyset,\{1\},\{2\}\}$, in a sense completing the work by Claesson and
Mansour~\cite{vincularpair}. We do not consider the following cases: 1) The case
where $X = Y = \emptyset$ since this is classical avoidance of two length $3$
patterns, which was done by Simion and Schmidt~\cite{classicsubset}; 2) The case
when either $X$ or $Y$ contains $0$ or $3$ since this forces the patterns to
occur with the first or last, smallest or largest, letters in the permutation
and would introduce too many cases to be considered in a single paper; 3) The
case when $X$ or $Y$ contain $\{1,2\}$ since these would be more naturally
treated in the context of consecutive patterns, see e.g., Elizalde and
Noy~\cite{ElizaldeNoy}.

An important property of the set of bivincular patterns, as noted by
Bousquet-M\'{e}lou et al.~\cite{bivincular}, is that it is closed under the
symmetries of the square. The set of patterns we are interested in---the union
of vincular and covincular patterns---is also closed under these symmetries, and
we make the following observation.

\begin{lemma}
  Let $\pi$ be a permutation and $p$ be a pattern, and let $\pi^*$ and $p^*$ be
  the permutation and pattern with the same symmetry applied to both $\pi$ and
  $p$. Then $\pi$ avoids $p$ if and only if $\pi^*$ avoids $p^*$.
\end{lemma}

From this we immediately see that if we can find the enumeration of
$\Av(p,r)$, for a single pair of patterns $p$ and $r$, then we
automatically have the enumeration for up to $8$ other symmetric cases. This
reduces the amount of work to be done considerably. In particular, we
need only consider when $Y \neq \emptyset$ as otherwise we could take a
symmetry to a case where instead $X = \emptyset$. Let
\begin{align*}
  \P &= \Bigl\{\,(\sigma,X,\emptyset) \,:\,
       \sigma\in\sym_3,\,X\in \bigl\{\emptyset,\{1\},\{2\}\bigr\}\,\Bigr\}; \\
  \Q &= \Bigl\{\,(\sigma,\emptyset,Y) \,:\,
       \sigma\in\sym_3,\,Y\in \{\{1\},\{2\}\bigr\}\,\Bigr\}
\end{align*}
In total we have $|\PQ| = (3!\cdot 3)\cdot (3!\cdot 2) = 216$ pairs of
patterns to consider. In Table~\ref{table:summary} we summarize our
results on permutations avoiding a pair of patterns from $\PQ$.

\renewcommand{\arraystretch}{1.3}
\begin{table}[htbp]
\centering
\begin{tabular}{| c | c | c |}
\hline
Enumeration                                                      & \# pairs & OEIS \\ %[1ex]
\hline
$C_n$                                                            & 24       & \seqnum{A000108} \\ %[1ex]
\hline
$\binom{n}{2} + 1$                                               & 16       & \seqnum{A000124} \\ %[1ex]
\hline
$2^{n-1}$                                                        & 104      & \seqnum{A000079} \\ %[1ex]
\hline
$\sum\limits_{i=0}^{n} \dbinom{\binom{i+1}{2}}{n-i}$             & 8        & \seqnum{A121690} \\%[1ex]
\hline
$\sum\limits_{k=0}^n \dbinom{\binom{k+1}{2} + n - k - 1}{n - k}$ & 8        & \seqnum{A098569} \\%[1ex]
\hline
$M_n$                                                            & 16       & \seqnum{A001006} \\ %[1ex]
\hline
OGF: $1+ \sum\limits_{n \geq 0} x^{n+1} L_n(1+x)$                & 8        & \seqnum{A249560} \\%[1ex]
\hline
OGF: $1 + \frac{x}{1-x} \sum\limits_{n \geq 0} \sum\limits_{k = 0}^{n+1} x^{i+k} L_{n + 1,k}
  \left(\frac{1}{1-x}\right)$                                    & 8        & \seqnum{A249561} \\%[1ex]
\hline
a recurrence relation                                            & 8        & \seqnum{A249563} \\ %[1ex]
\hline
a recurrence relation                                            & 4        & \seqnum{A249562} \\ %[1ex]
\hline
finite                                                           & 12       & - \\ \hline
\end{tabular}
\caption{The number of permutations avoiding a pair of patterns in $\PQ$.
  }\label{table:summary}
\end{table}

In Table~\ref{table:summary}, $C_n$ and $M_n$ are the Catalan and Motzkin
numbers, respectively. The sequences \seqnum{A249560}--\seqnum{A249563} were
added to the OEIS~\cite{motzkin} by the authors. In \seqnum{A249560} and
\seqnum{A249561}, $L_n(q) = \sum_{m = 0}^n {n \brack m}_q$ and $L_{n,k}(q) =
q^{n+\binom{k}{2}} {n-1 \brack k-1}_q$ enumerate specific types of lattice paths
and their areas.

It is sometimes possible to show that avoiding a given pattern $p$ is equivalent
to avoiding a simpler pattern $p'$. The following lemma states three instances
of this that are used here. This lemma is part of a more general result called
the shading lemma, due to Hilmarsson \emph{et al.}~\cite[Lemma 3.11]{shading}.

We first need to introduce the idea of a mesh pattern. In our previous pictures
we shaded entire rows or columns. In a mesh pattern we can shade individual
squares in the diagram.  As an example, below is a mesh pattern with a single
square shaded:
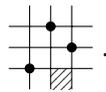
\[
    \mpattern{scale=\pattdispscale}{ 3 }{ 1/1, 2/3, 3/2 }{2/0}.
\]
A subsequence $\pi(i)\pi(j)\pi(k)$ of $\pi\in\sym_n$, that is order isomorphic
to $132$, is an occurrence of this particular pattern if there does not exist an
$m$ such that $j < m < k$ and $\pi(m) < \pi(i)$.  Mesh patterns satisfy a
property analogous to Remark~\ref{subpattern}: given a pattern $p = (\sigma,B)$,
where $B$ is the set of squares shaded, and $p' = (\sigma,B')$, where $B'
\subseteq B$, then $\Av_n(p') \subseteq \Av_n(p)$. The original definition of
mesh patterns was given by Br\"{a}nd\'{e}n and Claesson~\cite{mesh}.

\begin{lemma}\label{Shading1}
\begin{enumerate}[(i)]
\item\label{case-i}
  $\Av_n\left(\mpattern{scale=\patttextscale}{3}{1/1,2/3,3/2}{1/0,1/1,1/2,1/3}\right) =
   \Av_n\left(\mpattern{scale=\patttextscale}{3}{1/1,2/3,3/2}{}\right)
  $
\item\label{case-ii}
  $\Av_n\left(\mpattern{scale=\patttextscale}{3}{1/1,2/3,3/2}{2/0,2/1,2/2,2/3}\right) =
  \Av_n\left(\mpattern{scale=\patttextscale}{3}{1/1,2/3,3/2}{2/0}\right).
  $
\item\label{case-iii}
  $\Av_n\left(\mpattern{scale=\patttextscale}{3}{1/1,2/2,3/3}{1/0,1/1,1/2,1/3}\right) =
  \Av_n\left(\mpattern{scale=\patttextscale}{3}{1/1,2/2,3/3}{1/3}\right).
  $
\end{enumerate}
\end{lemma}

\begin{proof}
  For (i) see \cite[Lemma~2]{vincular} and for (ii) and (iii) see
  \cite[Lemma~3.11]{shading}.
\end{proof}

\begin{remark}
  It is important to note that
  $\Av_n(132,\{2\},\emptyset) \neq \Av_n(132,\emptyset,\emptyset)$. For
  example, $2413 \notin \Av_4(132,\emptyset,\emptyset)$ but
  $2413 \in \Av_4(132,\{2\},\emptyset)$.
\end{remark}

\section{Catalan numbers (\texorpdfstring{\seqnum{A000108}}{A000108})}\label{classicsection}

There are 24 pairs $(p,r)\in\PQ$ such that $|\Av_n(p,r)|=C_n$. These break down
to five cases once symmetries are considered; see Table~\ref{big-table} in the
Appendix for a full list. They can all be simplified to the avoidance of a
single classical pattern. We look at a case for each argument. By
Remark~\ref{subpattern} we have:

\begin{proposition}\label{shadingonce}
  $\Av_n\left(
    \mpattern{scale=\patttextscale}{3}{1/1,2/2,3/3}{},
    \mpattern{scale=\patttextscale}{3}{1/1,2/2,3/3}{0/2,1/2,2/2,3/2}\right)
  =\Av_n(123)$.
\end{proposition}

Similarly, by first using Lemma \ref{Shading1}\ref{case-i} on the first pattern
and then using Remark~\ref{subpattern} on the resulting pair we have:

\begin{proposition}\label{shadingtwice}
  $\Av_n\left(
    \mpattern{scale=\patttextscale}{ 3 }{ 1/1, 2/3, 3/2 }{1/0, 1/1, 1/2, 1/3},
    \mpattern{scale=\patttextscale}{ 3 }{ 1/1, 2/3, 3/2 }{0/1, 1/1, 2/1, 3/1}
  \right) = \Av_n(132)$.
\end{proposition}

It is well known that the enumeration for avoiding any classical pattern of
length $3$ is given by the Catalan numbers; see Knuth \cite[Section 2.2.1,
Exercises 4 and 5]{classical}. Thus the sets in the propositions above all have
cardinality $C_n$.  In Section \ref{avoiding123} we present a new proof that
$|\Av_n(123)| = C_n$.

The rest of the pairs that are counted by the Catalan numbers all follow a very
similar argument so we move on to the next case.

\section{Central polygonal numbers (\texorpdfstring{\seqnum{A000124}}{A000124})}\label{central-polygonal}

After considering symmetries there are three pairs $(p,r)\in\PQ$ such that
$|\Av_n(p,r)|=\binom{n}{2} + 1$; see Table~\ref{big-table} in the Appendix. They
all reduce to the already known case $|\Av_n(123,231)| = \binom{n}{2} + 1$ done
by Simion and Schmidt~\cite{classicsubset}.

\begin{proposition}
  $\Av_n\left(
    \mpattern{scale=\patttextscale}{ 3 }{ 1/1, 2/2, 3/3 }{},
    \mpattern{scale=\patttextscale}{ 3 }{ 1/2, 2/3, 3/1 }{ 0/1, 1/1, 2/1, 3/1 }
  \right)=\Av_n(123,231)$.
\end{proposition}

\begin{proof}
  Use symmetry and Lemma~\ref{Shading1}\ref{case-i}.
\end{proof}

\begin{proposition}\label{shadeonebox}
  $\Av_n\left(
    \mpattern{scale=\patttextscale}{ 3 }{ 1/2, 2/3, 3/1 }{},
    \mpattern{scale=\patttextscale}{ 3 }{ 1/1, 2/2, 3/3 }{ 0/1, 1/1, 2/1, 3/1 }
  \right)=\Av_n(123,231)$.
\end{proposition}

\begin{proof}
  After applying Lemma \ref{Shading1} we see that the set we are interested in
  is
  \[
  \mathcal{B} = \Av\left( \mpattern{scale=\patttextscale}{ 3 }{ 1/2, 2/3, 3/1 }{},
    \mpattern{scale=\patttextscale}{ 3 }{ 1/1, 2/2, 3/3 }{ 3/1 }
  \right).
  \]
  We want to show that $\mcB$ is equal to
  $\mcA=\Av(123,231)$.  It is clear than
  $\mcA\subseteq \mcB$. We will show
  $\mcB\subseteq \mcA$ by contraposition. Assume that
  $\pi\notin \mcA$.  If $\pi$ contains $231$ then it immediately
  follows that $\pi\notin \mcB$. If $\pi$ contains $123$ then
  either $\pi$ contains
  \[
    \mpattern{scale=\pattdispscale}{ 3 }{ 1/1, 2/2, 3/3 }{ 3/1 }
  \] in
  which case $\pi\notin\mcB$, or else it contains the pattern
  where there is a point in the shaded square. This would create an
  occurrence of $1342$ which contains an occurrence of $231$ and hence
  we would have $\pi\notin\mcB$. Therefore
  $\mcA = \mcB$.
\end{proof}

\begin{proposition}
  $\Av_n\left(
    \mpattern{scale=\patttextscale}{ 3 }{ 1/2, 2/3, 3/1 }{ 2/0, 2/1, 2/2, 2/3 },
    \mpattern{scale=\patttextscale}{ 3 }{ 1/1, 2/2, 3/3 }{ 0/1, 1/1, 2/1, 3/1 }
  \right)=\Av_n(123,231)$.
\end{proposition}

\begin{proof}
  Apply Lemma~\ref{Shading1}\ref{case-i} and then Proposition~\ref{shadeonebox}.
\end{proof}

\section{Powers of 2 (\texorpdfstring{\seqnum{A000079}}{A000079})}\label{structuresection}

After considering symmetries there are 19 different pairs $(p,r)\in\PQ$ such
that $\Av_n(p,r)=2^{n-1}$ and we list them in Table~\ref{big-table} of
the Appendix. Most of the cases reduce to
\[
    |\Av_n(123,132)| = |\Av_n(132,312)| = |\Av_n(231,312)| = 2^{n-1},
\]
as shown by Simion and Schmidt \cite{classicsubset}. There are two non-trivial
cases where we use the structure of the set to find a generating function for
the enumeration.

\begin{proposition}\label{2nfirst}
  Let $p=\mpattern{scale=\patttextscale}{3}{1/1,2/2,3/3}{2/0,2/1,2/2,2/3}$ and
  $r = \mpattern{scale=\patttextscale}{3}{1/3,2/1,3/2}{0/2,1/2,2/2,3/2}$.
  The number of permutations in
  \[
  \Av_n(p, r) \quad\text{is}\quad 2^{n-1}, \quad\text{for}\quad n \geq 1.
  \]
\end{proposition}

\begin{proof}
  Let $\mcA$ be the set of avoiders in question and let
  $\pi\in\mcA$. As $\pi$ avoids $p$, the points after the minimum of
  $\pi$ form a decreasing sequence. Moreover, in order to ensure that the
  permutation avoids $r$, every point to the right of the minimum must be
  greater than every point on the left of the minimum. Therefore, all non-empty
  permutations of $\mcA$ have the form
  \[
    \metapatt{0.6}{1/2}{1/v,1/h,2/h}{}{1/1/p}[0/0,1/0,0/2,1/1]
    [0/1/1/2/$\mcA$,1/2/2/3/\decr]
  \]
  where $\mcA$ symbolizes a (possibly empty) permutation which avoids the
  patterns, and \,\raisebox{3pt}{\scalebox{0.6}{$\decr$}}\, symbolizes a
  decreasing permutation. As the structure is so rigid we can find the ordinary
  generating function of the avoiders by multiplying together the ordinary
  generating functions of the component parts. There is one decreasing
  permutation of length $n$ and so the ordinary generating function is
  $1/(1-x)$. The ordinary generating function of a single point is $x$. If $A$
  is the ordinary generating function for $\mcA$, then it follows that
  \[
    A = A \cdot x \cdot \frac{1}{1-x} + 1
  \]
  where we add $1$ for the empty permutation which trivially avoids both
  patterns. Rearranging we get
  \[
    A = \frac{1-x}{1-2x} = 1 + \sum_{n\geq 1} 2^{n-1}x^n. \qedhere
  \]
\end{proof}

\begin{proposition}\label{2nsecond}
  Let $p = \mpattern{scale=\patttextscale}{3}{1/1,2/2,3/3}{}$ and $r =
  \mpattern{scale=\patttextscale}{3}{1/3,2/1,3/2}{0/1,1/1,2/1,3/1}$. The
  number of permutations in
  \[
    \Av_n(p, r) \quad\text{is}\quad 2^{n-1}, \quad\text{for}\quad n \geq 1.
  \]
\end{proposition}

\begin{proof}
  Let $\mcA$ be the set of avoiders in question. Consider the leftmost point
  $\ell$ of a permutation in $\mcA$. To avoid $p$ the points greater than $\ell$
  must form a decreasing sequence and similarly to avoid $r$ the points less
  than $\ell$ must form decreasing sequence. Therefore the permutations in
  $\mcA$ have the form below.
  \[
    \metapatt{0.6}{1/1}{1/v,1/h}{}{1/1/p}[0/0,0/1][1/0/2/1/\decr,1/1/2/2/\decr][]
  \]
  A permutation matching this picture cannot contain an occurrence of $p=123$,
  and every occurrence of $312$ will have the point $\ell$ preventing it from
  being an occurrence of $r$. Hence these can be encoded with binary strings and
  so there are $2^{n-1}$ such permutations.
  %  if we let $A$ be the exponential generating
  % function for $\mcA$ then
  % \[
  %   A = 1 + \int e^x \cdot e^x dx = 1 + \frac{e^{2x}}{2} = 1 + \sum_{n\geq1}
  %   \frac{2^{n-1} x^n}{n!}. \qedhere
  % \]
\end{proof}

\section{Left-to-right minima boundaries (\texorpdfstring{\seqnum{A121690}}{A121690})}\label{lrmsection}
After symmetries there is exactly one pair $(p,r)\in\PQ$ enumerated by the
formula in the following proposition.
\begin{proposition}\label{choose}
  Let $p = \mpattern{scale=\patttextscale}{3}{1/1,2/2,3/3}{2/0,2/1,2/2,2/3}$ and
  $r = \mpattern{scale=\patttextscale}{3}{1/1,2/3,3/2}{0/2,1/2,2/2,3/2}$.
  The number of permutations in
  \[
      \Av_n(p,r)
      \quad\text{is}\quad
      \sum_{k=0}^n \binom{\binom{k+1}{2}}{n-k}.
  \]
\end{proposition}

\begin{figure}[ht]
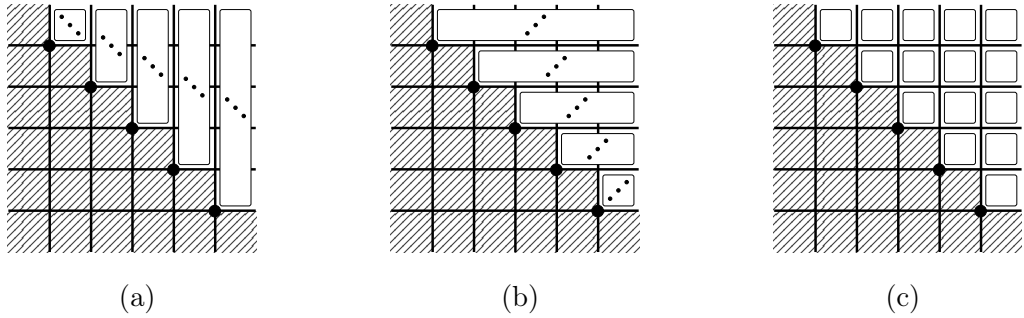

  \centering
  \mbox{
  \begin{subfigure}[t]{0.3\textwidth}
    \centering
    \metapatt{0.55}{5/5}{}{1/v,2/v,3/v,4/v,5/v,1/h,2/h,3/h,4/h,5/h}{1/5,2/4,3/3,4/2,5/1}
    [0/0,0/1,0/2,0/3,0/4,0/5,1/0,1/1,1/2,1/3,1/4,2/0,2/1,2/2,2/3,3/0,3/1,3/2,4/0,4/1,5/0]
    [1/5/2/6/\decr,2/4/3/6/\decr,3/3/4/6/\decr,4/2/5/6/\decr,5/1/6/6/\decr][]
    \caption{}\label{fig:struct-a}
  \end{subfigure}
  \begin{subfigure}[t]{0.3\textwidth}
    \centering
    \metapatt{0.55}{5/5}{}{1/v,2/v,3/v,4/v,5/v,1/h,2/h,3/h,4/h,5/h}{1/5,2/4,3/3,4/2,5/1}
    [0/0,0/1,0/2,0/3,0/4,0/5,1/0,1/1,1/2,1/3,1/4,2/0,2/1,2/2,2/3,3/0,3/1,3/2,4/0,4/1,5/0]
    [1/5/6/6/\incr,2/4/6/5/\incr,3/3/6/4/\incr,4/2/6/3/\incr,5/1/6/2/\incr][]
    \caption{}\label{fig:struct-b}
  \end{subfigure}
  \begin{subfigure}[t]{0.3\textwidth}
    \centering
    \metapatt{0.55}{5/5}{}{1/v,2/v,3/v,4/v,5/v,1/h,2/h,3/h,4/h,5/h}{1/5,2/4,3/3,4/2,5/1}
    [0/0,0/1,0/2,0/3,0/4,0/5,1/0,1/1,1/2,1/3,1/4,2/0,2/1,2/2,2/3,3/0,3/1,3/2,4/0,4/1,5/0]
    [1/5/2/6/,2/5/3/6/,3/5/4/6/,4/5/5/6/,5/5/6/6/,2/4/3/5/,3/4/4/5/,4/4/5/5/,5/4/6/5/,
    3/3/4/4/,4/3/5/4/,5/3/6/4/,4/2/5/3/,5/2/6/3/,5/1/6/2/][]
    \caption{}\label{fig:struct-c}
  \end{subfigure}
  }
  \caption{The structure of $\Av_n(p,r)$ from Proposition~\ref{choose}}
\end{figure}

In the following proof we will consider the set of left-to-right minima, which
we call the \emph{boundary} of the permutation.
In later sections we will consider other types of boundaries that use
left-to-right minima, right-to-left minima or right-to-left maxima and perhaps
unions of these. It should be clear from context which type of boundary it is.

\begin{proof}
  Consider the minimum point, $1$, of a permutation in $\Av_n(p,r)$. From the
  pattern $p$ we see that the points to the right of $1$ form a decreasing
  sequence. Moreover, the points vertically between any two adjacent points on
  the boundary must form a decreasing sequence, giving the structure in
  Figure~\ref{fig:struct-a}, where the permutation we have drawn has five
  left-to-right minima.

  Now, consider the leftmost point, say $\ell$, of a permutation in
  $\Av_n(p,r)$. From the pattern $r$ we see that the points greater than $\ell$
  must form an increasing sequence. Moreover, considering the points
  horizontally between any two adjacent points on the boundary we get the
  structure in Figure~\ref{fig:struct-b}, where, again, the permutation has five
  left-to-right minima. When we overlay the given conditions above we get the
  structure in Figure~\ref{fig:struct-c}, where each of the squares in the
  diagram must be both increasing and decreasing. Therefore each square must be
  empty or contain a single point. Also, the structure of the rows and columns
  will be determined as increasing and decreasing, respectively, no matter which
  squares have points. Therefore, placing any number of points into the squares
  (at most one in each) will create a unique permutation (see
  Figure~\ref{fig:example}).
  \begin{figure}[ht]
    \centering
    \begin{tikzpicture}[scale=0.55]
      \draw[black, line width=0.8]
        (1,3)--(2,3)--(2,2)--(3,2)--(3,1)
        (1,4)--(1,3)
        (3,1)--(4,1)
        (4,2)--(3,2)--(3,4)
        (4,3)--(2,3)--(2,4);
      % draw the base permutation points
      \foreach [count=\i] \x in {3,2,1}
        \filldraw(\i*1, \x) circle (0.12);
      % draw additional points
      \filldraw(1.5,3.25) circle (0.08);
      \filldraw(2.33,3.5) circle (0.08);
      \filldraw(3.25,3.75) circle (0.08);
      \filldraw(2.66,2.33) circle (0.08);
      \filldraw(3.5,2.66) circle (0.08);
      \filldraw(3.75,1.5) circle(0.08);
    \end{tikzpicture}
    \caption{The permutation $673841952 \in \Av_9(p,r)$ from Proposition~\ref{choose}}
    \label{fig:example}
  \end{figure}
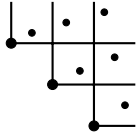
  Create a permutation $\pi \in \Av_n(p, r)$ with $k$ left-to-right
  minima. We need to know how to place the remaining $n - k$ points. There will
  be $\binom{k+1}{2}$ squares available to choose from, and placing the $n-k$
  points into any subset of those squares will create a unique permutation.
  Thus, summing over the number of left-to-right minima, we get
  \[
    |\Av_n(p,r)| = \sum_{k=0}^n \binom{\binom{k+1}{2}}{n-k}. \qedhere
  \]
\end{proof}

We note from the above proof and the formula that a
permutation with $k$ left-to-right minima will be of length at most
$k + \binom{k+1}{2}$. Also, a length $n$ avoider will have at least
$\left\lceil{\frac{{-3 + \sqrt{9 + 8n}}}{2}}\right\rceil$ left-to-right minima.

\section{Barred patterns (\texorpdfstring{\seqnum{A098569}}{A098569})}\label{barsection}

After considering symmetries there are two pairs $(p,r) \in \PQ$ enumerated by
the formula in the following proposition.

\begin{proposition}\label{barred1}
  Let $p = \mpattern{scale=\patttextscale}{3}{1/1,2/2,3/3}{2/0,2/1,2/2,2/3}$ and
  $r = \mpattern{scale=\patttextscale}{3}{1/1,2/2,3/3}{0/2,1/2,2/2,3/2}$. The
  number of permutations in
  \[
  \Av_n(p,r)
  \quad\text{is}\quad
  \sum_{k=0}^n \binom{\binom{k+1}{2} + n - k - 1}{n - k}.
  \]
\end{proposition}

\begin{proof}
  Consider the left-to-right minima of a permutation $\pi \in \Av_n(p,r)$ as we
  did in Proposition~\ref{choose}. The points vertically between any two
  adjacent left-to-right minima must form a decreasing sequence and the points
  horizontally between any two adjacent left-to-right minima must also form a
  decreasing sequence. If we overlay these two conditions we get a structure
  like that in Figure~\ref{fig:struct-c}, where, in this case, each of the
  squares in the diagram must be decreasing. Also, the structure of the rows and
  columns will be determined as decreasing no matter which squares have points.
  Therefore, placing any number of points into the squares will create a unique
  permutation, and so the ordinary generating function for $|\Av_n(p,r)|$ is
  \[
    \sum_{k\geq0} x^k \left(\frac{1}{1-x}\right)^{\binom{k+1}{2}}.
  \]
  The coefficient of $x^j$ in $1/(1-x)^m$ is $\binom{m+j-1}{j}$ which concludes
  the proof.
\end{proof}

It is possible to show that avoiding the two patterns $p$ and $r$, above, is
equivalent to avoiding a single \emph{barred} pattern first introduced by
West~\cite{westbar}. For a more detailed account of barred patterns see
Pudwell~\cite{barpatt}. The following is the definition which can be found in
that reference.

\begin{definition}{(Pudwell~\cite[p.~1]{barpatt})}
  Let $p$ be a barred pattern. Let $r$ be the pattern with the bars removed and
  let $p'$ be the permutation order isomorphic to the pattern in which the
  barred numbers are removed. We say that a permutation contains $p$ if every
  occurrence of $p'$ can be extended to an occurrence of $r$.
\end{definition}

For example, a permutation contains $\bar{4}25\bar{1}3$ if every occurrence of
$132$ is contained as $253$ in a $42513$ pattern. Barred patterns can often be
thought of as mesh patterns (see Ulfarsson~\cite[p.~5]{west3}). For instance,
\[
  \Av_n(\bar{4}23\bar{1}5) =
  \Av_n\left(
  \mpattern{scale=\patttextscale}{3}{1/1,2/2,3/3}{2/0},
  \mpattern{scale=\patttextscale}{3}{1/1,2/2,3/3}{0/2}\right)
  =
  \Av_n\left(\mpattern{scale=\patttextscale}{3}{1/1,2/2,3/3}{2/0,2/1,2/2,2/3},
    \mpattern{scale=\patttextscale}{3}{1/1,2/2,3/3}{0/2,1/2,2/2,3/2}\right)
  \]
and
\[
  \Av_n(\bar{4}25\bar{1}3) =
  \Av_n\left(
    \mpattern{scale=\patttextscale}{3}{1/1,2/3,3/2}{2/0},
    \mpattern{scale=\patttextscale}{3}{1/1,2/3,3/2}{0/2}\right) =
  \Av_n\left(
    \mpattern{scale=\patttextscale}{3}{1/1,2/3,3/2}{2/0,2/1,2/2,2/3},
    \mpattern{scale=\patttextscale}{3}{1/1,2/3,3/2}{0/2,1/2,2/2,3/2}\right),
\]
where the second equalities in both equations follow from Lemma~\ref{Shading1}
\begin{corollary}
  The number of permutations in $\Av_n\left(\bar{4}23\bar{1}5\right)$ is
  \[
    \sum_{k = 0}^n \binom{ \binom{k + 1}{2} + n - k - 1}{n - k}.
  \]
\end{corollary}

This confirms the conjecture from Pudwell \cite[page 8]{barpatt} that
$\bar{4}25\bar{1}3$ and $\bar{4}23\bar{1}5$ are Wilf-equivalent, i.e.,~that the
number of permutations avoiding either one is the same. If we were to apply the
same method as in the proof of Proposition~\ref{barred1} to
\[
\Av_n\left(
  \mpattern{scale=\patttextscale}{3}{1/1,2/3,3/2}{2/0,2/1,2/2,2/3},
  \mpattern{scale=\patttextscale}{3}{1/1,2/3,3/2}{0/2,1/2,2/2,3/2}\right)
= \Av_n(\bar{4}25\bar{1}3)
\]
then we would have a similar structure with the left-to-right minima, where,
however, we get increasing sequences in the squares and along the rows
and columns.

\section{Motzkin numbers (\texorpdfstring{\seqnum{A001006}}{A001006})}\label{motzkinsection}

The Motzkin numbers, $M_n$, form a well known sequence which can be
defined by a functional equation their ordinary generating function
satisfies:
\[
M = 1 + xM + x^2M^2\quad\text{where}\quad M=\sum_{n\geq 0} M_n x^n.
\]
For more information on Motzkin numbers see
e.g.,~OEIS~\cite{actualmotzkin}.  After considering symmetries and Lemma
\ref{Shading1} we have two cases such that $|\Av_n(p,r)|=M_n$ and
$(p,r)\in\PQ$. For a full list see Table~\ref{big-table} of the
Appendix.

\begin{proposition}[Elizalde and Mansour~\cite{motzkinbijection}]\label{motzkingen}
  Let $p = \mpattern{scale=\patttextscale}{3}{1/1,2/3,3/2}{}$ and $r =
  \mpattern{scale=\patttextscale}{3}{1/1,2/2,3/3}{0/2,1/2,2/2,3/2}$. The number
  of permations in
  \[
    \Av_n(p,r) \quad\text{is}\quad M_n.
  \]
\end{proposition}

The proof given by Elizalde and Mansour~\cite{motzkinbijection} provides a
bijection between $\Av_n(p,r)$ and Motzkin paths. We show that the structure of
the permutations implies they are enumerated by the Motzkin numbers.

\begin{proof}
  Consider a permutation $\pi$ in $\mcA = \Av(p,r)$. Further, consider
  the rightmost point of $\pi$. For $\pi$ to avoid $p$ the structure of $\pi$
  must be like Figure~\ref{fig:Motzin-struct-a}.
  \begin{figure}[ht]
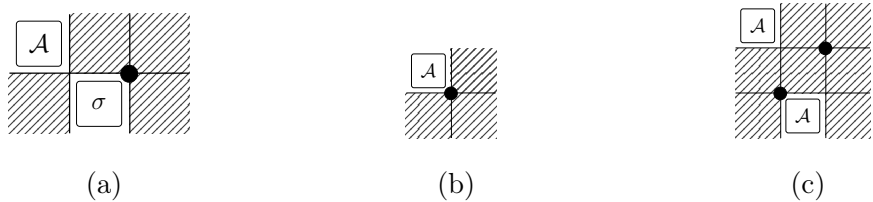

    \centering
    \mbox{
    \begin{subfigure}[b]{0.3\textwidth}
      \centering
      \metapatt{0.8}{2/1}{1/v,2/v,1/h}{}{2/1/p}
      [0/0,2/0,1/1,2/1]
      [0/1/1/2/$\mcA$,1/0/2/1/$\sigma$]
      \caption{}\label{fig:Motzin-struct-a}
    \end{subfigure}
    \begin{subfigure}[b]{0.25\textwidth}
      \centering
      \metapatt{0.6}{1/1}{1/v,1/h}{}{1/1/p}[0/0,1/0,1/1,][0/1/1/2/$\mcA$]
      \caption{}\label{fig:Motzin-struct-b}
    \end{subfigure}
    \begin{subfigure}[b]{0.3\textwidth}
      \centering
      \metapatt{0.6}{2/2}{1/v,2/v,1/h,2/h}{}{1/1/p,2/2/p}[0/0,0/1,1/1,1/2,2/0,2/1,2/2]
      [0/2/1/3/$\mcA$,1/0/2/1/$\mcA$]
      \caption{}\label{fig:Motzin-struct-c}
    \end{subfigure}
    }
    \caption{The structure of $\Av(p,r)$ from Proposition~\ref{motzkingen}}
  \end{figure}
  With regard to $\sigma$, let us consider two cases. Either $\sigma$ is empty
  or it has at least one point. If $\sigma$ is empty the structure looks like
  Figure~\ref{fig:Motzin-struct-b}.  If $\sigma$ is non-empty then consider the
  maximum point, $m$, of $\sigma$. If there was a point to the left of $m$ in
  $\sigma$ then this point together with $m$ and the rightmost point of $\pi$
  would create an occurrence of $r$. Therefore there must be no points to the
  left of $m$ in $\sigma$. Thus we can place any possibly empty smaller
  permutation in $\mcA$ to the right of the maximum of $\sigma$ without
  creating an occurrence of $p$ or $r$, and so we have the structure in
  Figure~\ref{fig:Motzin-struct-c}.

  In conclusion, any non-empty permutation in $\mcA$ either has a
  structure described by Figure~\ref{fig:Motzin-struct-b} or a structure
  described by Figure~\ref{fig:Motzin-struct-c}. Letting $A$ denote the ordinary
  generating function for $\mcA$ we thus have $A = 1 + xA + x^2A^2$, from
  which the claim follows.
\end{proof}

We now go on to the second case. We will consider the structure of the avoiders
in terms of the left-to-right minima, as in Proposition~\ref{choose}.

\begin{proposition}\label{motzkinrecur}
  Let $p = \mpattern{scale=\patttextscale}{3}{1/1,2/2,3/3}{}$ and
  $r = \mpattern{scale=\patttextscale}{3}{1/1,2/3,3/2}{0/2,1/2,2/2,3/2}$.
  The number of permutations in
  \[
  \Av_n(p,r) \quad\text{is}\quad M_n.
  \]
\end{proposition}

\begin{proof}
  Let $\pi\in\Av_n(p,r)$ and consider the boundary of (the diagram of) $\pi$
  given by the left-to-right minima.  As in Proposition~\ref{choose}, any cell
  in the diagram must be both increasing and decreasing and so the cell is empty
  or contains exactly one point. Because the rows are increasing and $\pi$
  avoids $123$ there can be at most one point in each row. Moreover, if there is
  a point in a cell then we cannot place a point in a cell further to the right
  and above.

  Pick the leftmost point in the leading diagonal of this grid. The points above
  this cell will then form a subword of $\pi$ which is of shorter length and
  also avoids both patterns. The points below this cell will similarly form a
  subword which avoids both patterns; see Figure~\ref{fig:Motzkin1}.
  \begin{figure}[ht]
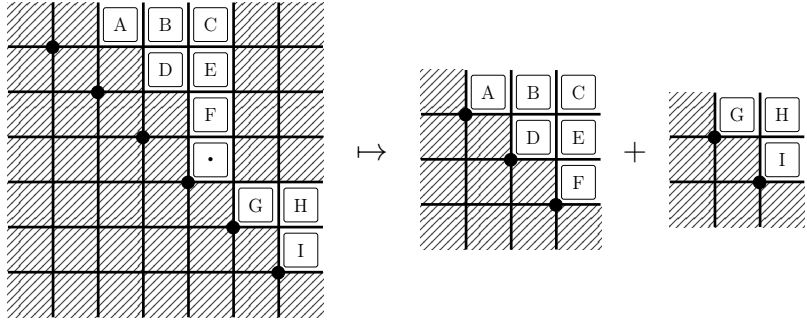

  \[
      \metapatt{0.6}{6/6}{}{1/v,2/v,3/v,4/v,5/v,6/v,1/h,2/h,3/h,4/h,5/h,6/h}
      {1/6,2/5,3/4,4/3,5/2,6/1}
      [0/0,0/1,0/2,0/3,0/4,0/5,0/6,1/0,1/1,1/2,1/3,1/4,1/5,2/0,2/1,2/2,2/3,2/4,3/0,3/1,
      3/2,3/3,4/0,4/1,4/2,5/0,5/1,6/0,6/3,6/4,6/5,6/6,5/3,5/4,5/5,5/6,3/4,2/5,1/6]
      [2/6/3/7/A,3/6/4/7/B,4/6/5/7/C,3/5/4/6/D,4/5/5/6/E,4/4/5/5/F,4/3/5/4/\Huge{$\cdot$},
      5/2/6/3/G,6/2/7/3/H,6/1/7/2/I][]
      \;\;\mapsto\;\;
      \metapatt{0.6}{3/3}{}{1/v,2/v,3/v,1/h,2/h,3/h}{1/3,2/2,3/1}
      [0/0,0/1,0/2,0/3,1/0,1/1,1/2,2/0,2/1,3/0]
      [1/3/2/4/A,2/3/3/4/B,3/3/4/4/C,2/2/3/3/D,3/2/4/3/E,3/1/4/2/F]
      \;+\;
      \metapatt{0.6}{2/2}{}{1/v,2/v,1/h,2/h}{1/2,2/1}
      [0/0,0/1,0/2,1/0,1/1,2/0][1/2/2/3/G,2/2/3/3/H,2/1/3/2/I]
  \]
  \caption{Decomposing a permutation in $\Av_n(p,r)$ from Proposition~\ref{motzkinrecur}}\label{fig:Motzkin1}
  \end{figure}
  Notice that this process is reversible: we can take a pair of avoiders of
  lengths $k$ and $n-k-2$, respectively, and glue them together by adding the
  leftmost point on the leading diagonal and the corresponding left-to-right
  minimum in this way.

  There is also the case when there are no points on the leading diagonal
  directly to the right of a left-to-right minima. In this case we remove the
  minimum point and tuck the points in the same way we did for the top half of
  the previous case, producing an avoider which is shorter in length; see
  Figure~\ref{fig:Motzkin2}.
  \begin{figure}[ht]
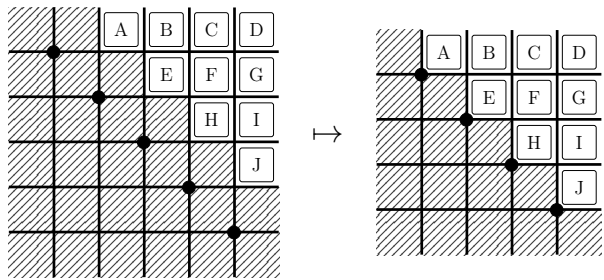

  \[
      \metapatt{0.6}{5/5}{}{1/v,2/v,3/v,4/v,5/v,1/h,2/h,3/h,4/h,5/h}
      {1/5,2/4,3/3,4/2,5/1}
      [0/0,0/1,0/2,0/3,0/4,0/5,1/0,1/1,1/2,1/3,1/4,1/5,2/0,2/1,2/2,2/3,2/4,3/0,
      3/1,3/2,3/3,4/0,4/1,4/2,5/0,5/1][2/5/3/6/A,3/5/4/6/B,4/5/5/6/C,5/5/6/6/D,
      3/4/4/5/E,4/4/5/5/F,5/4/6/5/G,4/3/5/4/H,5/3/6/4/I,5/2/6/3/J][] \;\;\mapsto\;\;
      \metapatt{0.6}{4/4}{}{1/v,2/v,3/v,4/v,1/h,2/h,3/h,4/h}{1/4,2/3,3/2,4/1}
      [0/0,0/1,0/2,0/3,0/4,1/0,1/1,1/2,1/3,2/0,2/1,2/2,3/0,3/1,4/0]
      [1/4/2/5/A,2/4/3/5/B,3/4/4/5/C,4/4/5/5/D,2/3/3/4/E,3/3/4/4/F,
      4/3/5/4/G,3/2/4/3/H,4/2/5/3/I,4/1/5/2/J][]
  \]
  \caption{``Shortening'' a permutation in $\Av_n(p,r)$ from Proposition~\ref{motzkinrecur}}\label{fig:Motzkin2}
  \end{figure}
  Again this is reversible, so we can take any length $n-1$ avoider and append a
  new minimum to create a length $n$ avoider. Thus, letting $A$ be the ordinary
  generating function for $\mcA$ we get that $A = 1 + xA + x^2A^2$.
\end{proof}

We will use a similar method to give a new proof of the enumeration of
$\Av_n(123)$ in Section~\ref{avoiding123}.

\section{Lattice paths and their area (\texorpdfstring{\seqnum{A249560}}{A249560})}\label{latticesection}
Up to symmetries there is a single pair $(p,r)$ in $\PQ$ with the enumeration
given in Proposition~\ref{lattice}, namely
\[
  \mcA_n = \Av_n\left(
  \mpattern{scale=\patttextscale}{3}{1/2,2/3,3/1}{1/0,1/1,1/2,1/3},
  \mpattern{scale=\patttextscale}{3}{1/1,2/3,3/2}{0/2,1/2,2/2,3/2}\right).
\]
To find the enumeration of this set we consider a different boundary than those
seen in previous sections. Our boundary here will be left-to-right minima and
right-to-left minima. We will first find a bijection between lattice paths and
the boundaries of permutations in $\mcA_n$. Then we extend this bijection by
considering the area under these paths.

For our purposes a \emph{lattice path} of length $n$ is a path that starts at
$(0,0)$ and has $n$ steps, each of which is
\begin{enumerate}
  \item[$N$]: $(x,y) \mapsto (x,y+1)$ or
  \item[$E$]: $(x,y) \mapsto (x+1,y)$.
\end{enumerate}
Clearly there are $2^n$ paths of length $n$. The following result is due to
Simion and Schmidt~\cite{classicsubset}, but we give a proof that is different
from theirs.

\begin{proposition}[Simion and Schmidt~\cite{classicsubset}]\label{bijection}
  There is a bijection between the length $n-1$ lattice paths and the
  permutations in $\Av_n(231,132)$.
\end{proposition}

\begin{proof}
  For $\pi \in \Av_n(231,132)$ define the path $w=x_nx_{n-1} \cdots x_2$ by
  \begin{equation*}
    x_k =
    \begin{cases}
      N & \text{if } \pi^{-1}(k) < \pi^{-1}(1) \text{ i.e.,~$k$ appears to the left of $1$},\\
      E & \text{otherwise}.
    \end{cases}
  \end{equation*}
  To see that $\pi\mapsto w$ is invertible note that the points to the left of
  the minimum of $\pi$ form a decreasing sequence, and, similarly, the points to
  the right of the minimum form an increasing sequence. Thus, any permutation
  $\pi\in\Av_n(231,132)$ is uniquely specified by the set $\{ i : \pi^{-1}(i) <
  \pi^{-1}(1)\}$ which coincides with the set $\{ i : x_i = N\}$.
\end{proof}

For example,
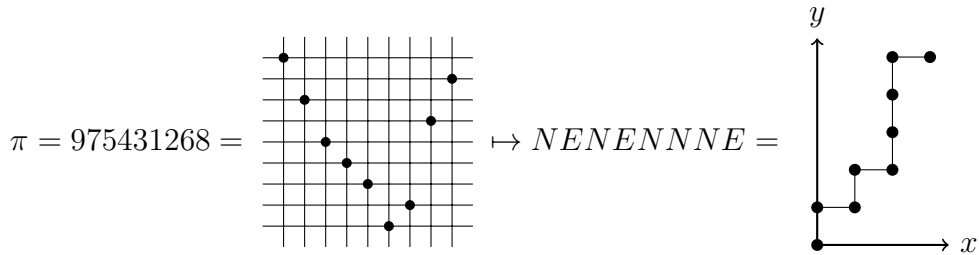
\[
  \pi = 975431268 =
  \mpattern{scale=\pattdispscale}{9}{1/9,2/7,3/5,4/4,5/3,6/1,7/2,8/6,9/8}{}\mapsto
  NENENNNE =
  \raisebox{-3.6em}{
    \begin{tikzpicture}[scale=0.5]
        \draw [<->,thick] (0,5.5) node (yaxis) [above] {$y$}
          |- (3.5,0) node (xaxis) [right] {$x$};
      \draw (0,0) [insert lattice path={ururuuur}];
    \end{tikzpicture}}
\]

Every lattice path defines an area enclosed by the path and the $x$-axis. We use
$q$-binomials to record this, see e.g.,~Azose~\cite{qbinomial} for more
details. In terms of $q$-binomial coefficients the number of length $n$ paths,
with $m$ $E$-steps, is given by ${n \brack m}_q$ where the coefficient of $q^k$
is the number of paths with area $k$. Let
\[
  L_n(q) = \sum_{m = 0}^n {n \brack m}_q,
\]
which is the distribution of area over all length $n$ paths. We will now
link this to pattern avoidance.

\begin{proposition}\label{lattice}
  Let $p = \mpattern{scale=\patttextscale}{3}{1/2,2/3,3/1}{1/0,1/1,1/2,1/3}$ and
  $r = \mpattern{scale=\patttextscale}{3}{1/1,2/3,3/2}{0/2,1/2,2/2,3/2}$.
  The ordinary generating function for
  \[
  \Av_n(p,r) \quad\text{is}\quad 1+ \sum_{n \geq 0} x^{n+1} L_n(1+x).
  \]
\end{proposition}

\begin{proof}
  Taking the left-to-right minima and right-to-left minima boundary of any
  permutation produces a permutation in the set $\Av_n(231,132)$. Consider a
  particular boundary of right-to-left minima and left-to-right minima of a
  permutation avoiding $p$ and $r$.  To avoid $p$ the points to the left of the
  minimum form a decreasing sequence and hence there are no other points in this
  region; also, the columns in between the right-to-left minima are forced to be
  increasing.  To avoid $r$, the rows in between the right-to-left minima must
  be decreasing, and the rows directly above a left-to-right minimum must be
  empty.  As an example, for the boundary given by $\pi = 975431268$ we get the
  following restrictions.
  \[
  \mpattern{scale=\pattdispscale}{9}{1/9,2/7,3/5,4/4,5/3,6/1,7/2,8/6,9/8}
  {0/0,0/1,0/2,0/3,0/4,0/5,0/6,0/7,0/8,0/9,1/0,1/1,1/2,1/3,1/4,1/5,1/6,1/7,1/8,
   1/9,2/0,2/1,2/2,2/3,2/4,2/5,2/6,2/7,2/8,2/9,3/0,3/1,3/2,3/3,3/4,3/5,3/6,3/7,
   3/8,3/9,4/0,4/1,4/2,4/3,4/4,4/5,4/6,4/7,4/8,4/9,5/0,5/1,5/2,5/3,5/4,5/5,5/6,
   5/7,5/8,5/9,6/0,6/1,6/2,6/6,6/8,7/0,7/1,7/2,7/3,7/4,7/5,7/6,7/8,8/0,8/1,8/2,
   8/3,8/4,8/5,8/6,8/7,8/8,9/0,9/1,9/2,9/3,9/4,9/5,9/6,9/7,9/8,9/9}
  \]
  In each of the unshaded squares we can place either a single point or leave it
  empty, and each such choice will create a unique permutation. The number of
  unshaded squares is given by the area under the lattice path corresponding to
  the boundary as in Proposition~\ref{bijection}. For example, in $\pi =
  975431268$ the three unshaded boxes in the top row correspond to the three
  squares in the bottom row of the area under its corresponding lattice path.
  This is because there are three left-to-right minima (excluding $1$), which
  ensures three $E$ steps in the lattice path, and $9$ appears at the beginning
  of $\pi$ ensuring a $N$ step at the start of the path.

  Hence, to count permutations in $\Av(p,r)$ we first fix the size of boundary,
  say $n+1$, giving the factor $x^{n+1}$. Then we substitute $q = 1 + x$ into
  $L_n(q)$, since this is the generating function for all length $n+1$
  boundaries with $q$ marking the squares that we can place a point in or leave
  empty.

\end{proof}

\section{Partitions into distinct parts (\texorpdfstring{\seqnum{A249561}}{A249561})}\label{partitionsection}
Up to symmetries there is a single pair $(p,r)$ in $\PQ$ with the enumeration
given in Proposition~\ref{partition}, namely
\[
\mcA = \Av(p,r) = \Av\left(
  \mpattern{scale=\patttextscale}{3}{1/1,2/2,3/3}{1/0,1/1,1/2,1/3},
  \mpattern{scale=\patttextscale}{3}{1/2,2/3,3/1}{0/2,1/2,2/2,3/2}\right).
\]
To find the enumeration of this set we consider the boundary of a permutation
$\pi \in \mcA$ given by its right-to-left maxima and right-to-left minima.
Taking this boundary of any permutation will result in a permutation that avoids
$231$ and $213$. Since avoiding $231$ implies avoiding $r$ by
Remark~\ref{subpattern} it is clear that any such boundary for a permutation in
$\mcA$ is in the set
\[
\mcA' = \Av \left(\mpattern{scale=\patttextscale}{3}{1/1,2/2,3/3}{1/0,1/1,1/2,1/3},
  \mpattern{scale=\patttextscale}{3}{1/2,2/3,3/1}{},
  \mpattern{scale=\patttextscale}{3}{1/2,2/1,3/3}{}\right).
\]
If we take $\pi' \in \mcA'$ and consider the restrictions implied by $p$ and
$r$, we see that the number of right-to-left maxima between the two rightmost
right-to-left minima does not change the number of unshaded squares (see
Figure~\ref{arbdec}).

\begin{figure}[ht]
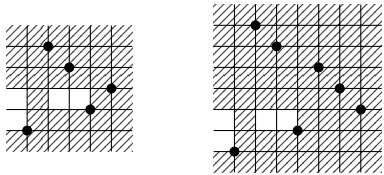

\centering
\mpattern{scale=\pattdispscale}{5}{1/1,2/5,3/4,4/2,5/3}{
0/0,0/3,0/4,0/5,
1/0,1/1,1/2,1/3,1/4,1/5,
2/0,2/1,2/3,2/4,2/5 ,
3/0,3/1,3/3,3/4,3/5,
4/0,4/1,4/2,4/3,4/4,4/5,
5/0,5/1,5/2,5/3,5/4,5/5}\phantom{---}
\mpattern{scale=\pattdispscale}{7}{1/1,2/7,3/6,4/2,5/5,6/4,7/3}{
0/0,0/3,0/4,0/5,0/6,0/7,
1/0,1/1,1/2,1/3,1/4,1/5,1/6,1/7,
2/0,2/1,2/3,2/4,2/5,2/6,2/7,
3/0,3/1,3/3,3/4,3/5,3/6,3/7,
4/0,4/1,4/2,4/3,4/4,4/5,4/6,4/7,
5/0,5/1,5/2,5/3,5/4,5/5,5/6,5/7,
6/0,6/1,6/2,6/3,6/4,6/5,6/6,6/7,
7/0,7/1,7/2,7/3,7/4,7/5,7/6,7/7}

\caption{The boundaries given by $15423$ and $1762543$}
\label{arbdec}
\end{figure}
We therefore start by considering $\pi' \in \mcA'$, where $\pi'(n-1) = \pi'(n) -
1$, i.e.~with a single right-to-left maximum after the rightmost left-to-right
minimum. In terms of pattern avoidance these boundaries are given by the set
\[
\mcB_n = \Av_n\left(
  \mpattern{scale=\patttextscale}{3}{1/1,2/2,3/3}{1/0,1/1,1/2,1/3},
  \mpattern{scale=\patttextscale}{3}{1/2,2/3,3/1}{},
  \mpattern{scale=\patttextscale}{3}{1/2,2/1,3/3}{},
  \mpattern{scale=\patttextscale}{2}{1/2,2/1}{1/0,1/1,1/2,2/0,2/1,2/2}\right)
  \subseteq \mcA_n'.
\]
Here the last pattern ensures that the condition $\pi'(n-1) = \pi'(n)-1$ is
enforced.

We will now show that the permutations in $\mcB_n$ are in bijection with
a subset of lattice paths.

\begin{definition}
  Let $w = x_1x_2 \cdots x_n$ be a lattice path. We say $w$ is a
  \emph{restricted lattice path} if
  \begin{enumerate}[(i)]
  \item  $x_1 = N$,
  \item  $x_n = E$ and
  \item  for all $i \in \{1, \ldots, n-1\}$ we have $x_ix_{i+1} \neq EE$.
  \end{enumerate}
  We define $\mathcal{R}_n$ to be the set of all restricted lattice paths of
  length $n$.
\end{definition}

\begin{remark}\label{rem:resttodist}
  A restricted lattice path, $w$, represents a unique integer partition since
  $w$ starts with an $N$ step and ends with an $E$ step. As there are no two
  consecutive $E$ steps in $w$ we can never have two columns of the same height,
  so it corresponds to an integer partition with distinct parts.
\end{remark}

\begin{proposition}\label{bijectionpart}
  There is a bijection between the restricted lattice paths in
  $\mathcal{R}_n$ and the permutations in $\mcB_n$.
\end{proposition}

\begin{proof}
  Let $\pi \in \mcB_n$. Define the path $w = Nx_1x_2 \cdots x_{n-1}$ by
  \begin{equation*}
    x_k =
    \begin{cases}
      N & \text{if } \pi^{-1}(k) > \pi^{-1}(n),\\
      E & \text{if } \pi^{-1}(k) < \pi^{-1}(n).
    \end{cases}
  \end{equation*}
  By definition the path $w$ starts with an $N$ step. Also, $\pi$ ends
  with an ascent, and so $x_{n-1} = E$. That $w$ does not contain $EE$
  can be seen by contraposition: Assume that $x_ix_{i+1}=EE$, then
  $\pi^{-1}(i)<\pi^{-1}(n)$ and $\pi^{-1}(i+1)<\pi^{-1}(n)$ which means
  that $\pi$ either contains the subsequence $(i,i+1,n)$ or it contains
  the subsequence $(i+1,i,n)$. In the latter case we have an occurrence
  of $213$ and we are done, so assume the former. If $i$ and $i+1$ are
  adjacent in $\pi$ then we have an occurrence of $p$. If not, then
  there must be a point in one of the lower three squares of the shading
  of $p$. But each of these three options leads to an occurrence of
  $p$ or $213$. This shows that the range of the mapping $\pi\mapsto w$
  is contained in $\mathcal{R}_n$. To see that $\pi\mapsto w$ is
  invertible we can reason in a way that is similar to the proof of
  Proposition~\ref{bijection}.
\end{proof}

\begin{remark}\label{points}
  Let $\lambda$ be the integer partition obtained from applying the
  above bijection to the permutation $\pi$. By
  Remark~\ref{rem:resttodist} it is clear that $\lambda$ has distinct
  parts. The number of points greater than $\pi(n)$ is one less than the
  maximum part of $\lambda$ and the number of points less than $\pi(n)$
  is the number of parts of $\lambda$. See Figure~\ref{fig:distinct} for
  an example.
\end{remark}

\begin{figure}[ht]
  \centering
  \mpattern{scale=\pattdispscale}{9}{1/9,2/1,3/8,4/2,5/7,6/6,7/5,8/3,9/4}{
    0/0,0/4,0/5,0/6,0/7,0/8,0/9,
    1/0,1/4,1/5,1/6,1/7,1/8,1/9,
    2/0,2/1,2/2,2/3,2/4,2/5,2/6,2/7,2/8,2/9,
    3/0,3/1,3/4,3/5,3/6,3/7,3/8,3/9,
    4/0,4/1,4/2,4/3,4/4,4/5,4/6,4/7,4/8,4/9,
    5/0,5/1,5/2,5/4,5/5,5/6,5/7,5/8,5/9,
    6/0,6/1,6/2,6/4,6/5,6/6,6/7,6/8,6/9,
    7/0,7/1,7/2,7/4,7/5,7/6,7/7,7/8,7/9,
    8/0,8/1,8/2,8/3,8/4,8/5,8/6,8/7,8/8,8/9,
    9/0,9/1,9/2,9/3,9/4,9/5,9/6,9/7,9/8,9/9}
  \;\;$\leftrightarrow$\;\;
  \raisebox{-4.25em}{
    \begin{tikzpicture}[scale = 0.45]
        \draw [<->,thick] (0,6.5) node (yaxis) [above] {$y$}
          |- (3.5,0) node (xaxis) [right] {$x$};
      \draw (0,0) [insert lattice path={uururuuur}];
    \end{tikzpicture}}
  \;\;$\leftrightarrow$\;\;
  $2 + 3 + 6$
  \caption{The boundary given by the permutation $918276534$ with the
    corresponding lattice path and integer partition with distinct
    parts}\label{fig:distinct}
\end{figure}

Partitions are well studied objects (see e.g.,~Andrews~\cite{partition}) and it
can be shown that if $q$ keeps track of the sum of the parts then the number of
partitions with maximum part $n$ into $k$ distinct parts is given by
\[
  L_{n,k}(q) = q^{n+\binom{k}{2}} {n-1 \brack k-1}_q.
\]

\begin{proposition} \label{partition}
  Let $p = \mpattern{scale=\patttextscale}{3}{1/1,2/2,3/3}{1/0,1/1,1/2,1/3}$ and
  $r = \mpattern{scale=\patttextscale}{3}{1/2,2/3,3/1}{0/2,1/2,2/2,3/2}$.
  The ordinary generating function for
  \[
    \mcA = \Av(p,r) \quad\text{is}\quad
    1+\frac{x}{1-x}\sum_{n \geq 0}\sum_{k=0}^{n+1}x^{i+k}L_{n+1,k}\left(\frac{1}{1-x}\right).
  \]
\end{proposition}

\begin{proof}
  Let $\pi\in\mcA$.  Consider the boundary given by right-to-left maxima
  and right-to-left minima.  As before we assume that $\pi(n-1) = \pi(n) - 1$
  and thus the boundary is in $\mathcal{B}_n$. To avoid $r$ the points above
  $\pi(n)$ must form a decreasing sequence. There must also be no points between
  a right-to-left minimum and right-to-left minimum in order to avoid $p$.  In
  the remaining unshaded regions, columns are decreasing (to avoid $p$) and rows
  are decreasing (to avoid $r$). Thus, an unshaded square can contain a
  decreasing sequence of any length. The bijection in
  Proposition~\ref{bijectionpart} gives a bijection that defines the available
  squares, and, considering Remark~\ref{points}, it follows that the ordinary
  generating function for $\mcA$ is as claimed.
\end{proof}

\section{Recurrence relations (\texorpdfstring{\seqnum{A249562}}{A249562} \& \texorpdfstring{\seqnum{A249563}}{A249563})}\label{recurrences}

We enumerate the two remaining pairs $(p,r) \in \PQ$ with recurrence relations.

\subsection{A recurrence for \texorpdfstring{\seqnum{A249563}}{A249563}} \label{firstrecurrence}

The set we are enumerating is
\[
\Av_n(p,r) = \Av_n \left(
  \mpattern{scale=\patttextscale}{3}{1/1,2/3,3/2}{2/0,2/1,2/2,2/3},
  \mpattern{scale=\patttextscale}{3}{1/1,2/2,3/3}{0/1,1/1,2/1,3/1}
\right).
\]
Let $\pi \in \Av_n(p,r)$. Write $\pi=m_1T_1 m_2T_2 \cdots m_kT_k$
where $m_1$, $m_2$, \ldots, $m_k$ are the left-to-right minima of $\pi$, and
$T_1, T_2, \ldots, T_k$ are the remaining points in between the
minima.
To avoid $p$ each $T_i$ must be increasing. We call $m_iT_i$
the $i$th \emph{block} of $\pi$.

Assume that $\pi$ has an occurrence of the pattern
\!\raisebox{0.2em}{\mpattern{scale=\patttextscale}{2}{1/1,2/2}{0/1,1/1,2/1}}\!\!.
Because $\pi$ avoids $r$ there cannot be any points above and to the
right of this occurrence.  This motivates the following definition.

\begin{definition}
  For a permutation $\pi \in \Av_n(r)$, if there exists $i$ and $j$ such
  that $j > i$ and $\pi(i) = \pi(j) - 1$ then we call $\pi(j)$ a
  \emph{ceiling point}.
\end{definition}

Going back to analyzing the structure of $\pi \in \Av_n(p,r)$, notice
that if we remove the maximum, $n$, from $\pi$ then the resulting
permutation will be in $\Av_{n-1}(p,r)$. This gives us ground for a
recursion. Consider inserting $n$ into a permutation in
$\Av_{n-1}(p,r)$. Where we can place $n$ depends on several factors. Let
$a_{n,k,i,\ell}$ be the number of avoiders where $n$ is the length of the
permutation; $k$ is the number of blocks; $i$ is the block that the
maximum is in and $\ell$ is the block containing the leftmost ceiling
point; if there is no ceiling point then we let $\ell = 0$.

It is clear that we can have at most $n$ blocks and that $n$ cannot
occur after the leftmost ceiling point. Therefore if $n < k$ or $i > \ell$
(while $\ell > 0$) then $a_{n,k,i,\ell} = 0$. There is a unique length $n$
permutation with $n$ blocks, namely the decreasing one. The maximum is in
the first block (except when $n=0$, in which case there is no maximum so we let
$i = 0$), hence we have
\[
a_{n,n,1,0} = 1 = a_{0,0,0,0}.
\]
We have three cases to consider. The new maximum, $n$, is inserted to
become a ceiling point (this is when $i = \ell$); $n$ is inserted to
create a new block (when $i = 1$) or $n$ is inserted into an existing
block but is not a ceiling point.

We first consider inserting $n$ into a length $n-1$ permutation so as to
ensure $n$ is a ceiling point. It must be placed after the current
maximum but before the leftmost ceiling point. If the smaller
permutation has no ceiling point then we can freely insert $n$. Hence
\[
a_{n,k,\ell,\ell} = a_{n-1,k,1,0} + \sum_{j=1}^i \sum_{m=i+1}^k a_{n-1,k,j,m}.
\]

Now we consider inserting $n$ so it is not a ceiling point. We may
either create a new block (when $i = 1$) or place it into an already
existing block. Consider inserting it into an existing block, then it
cannot be placed after the current maximum or else it will become a
ceiling point. The leftmost ceiling point will carry over to the larger
permutation. Therefore, if $i < \ell$,
\[
a_{n,k,i,\ell} = \sum_{j=i+1}^k a_{n-1,k,i,\ell}.
\]

To create a new block we can add $n$ to any length $n-1$ avoider but there
will of course be a shift of indices. If $i = 1$ we get
\[
a_{n,k,i,\ell} = \sum_{j=i+1}^k a_{n-1,k,j,\ell} + \sum_{j=0}^{k-1} a_{n-1,k-1,j,\ell-1}.
\]
This, with the initial conditions, gives a recursion for $a_{n,k,i,\ell}$.

\begin{proposition} \label{prop:firstrecurrence}
  Let $p = \mpattern{scale=\patttextscale}{3}{1/1,2/3,3/2}{2/0,2/1,2/2,2/3}$ and
  $r = \mpattern{scale=\patttextscale}{3}{1/1,2/2,3/3}{0/1,1/1,2/1,3/1}$.
  The number of permutations in $\Av_n(p,r)$ is given by
  \[
  \left\{\sum_{k=0}^n \sum_{i=0}^k \sum_{\ell = 0}^k a_{n,k,i,\ell}\right\}_{\!n\geq 0}
  \!=\; \{1, 1, 2, 4, 9, 22, 57, 156, 447, 1335, 4140, \ldots \}.
  \]
\end{proposition}

\noindent
This sequence was added to the OEIS by the authors
\cite[\seqnum{A249563}]{motzkin}.

\subsection{A recurrence for \texorpdfstring{\seqnum{A249562}}{A249562}}\label{secondrecurrence}

Here we enumerate the set
\[
\Av_n(p,r) = \Av_n \left(
  \mpattern{scale=\patttextscale}{3}{1/1,2/2,3/3}{2/0,2/1,2/2,2/3},
  \mpattern{scale=\patttextscale}{3}{1/1,2/2,3/3}{0/1,1/1,2/1,3/1}
\right).
\]
Let $\pi \in \Av_n(p,r)$. Write $\pi=m_1T_1 m_2T_2 \cdots m_kT_k$ where
$m_1$, $m_2$, \ldots, $m_k$ are the left-to-right minima of $\pi$, and $T_1,
T_2, \ldots, T_k$ are the remaining points in between the minima. To avoid
$p$ each $T_i$ must be decreasing. We call $m_iT_i$ the $i$th \emph{block}
of $\pi$.  Notice that removing $n$ from $\pi$ will result in a permutation in
$\Av_{n-1}(p,r)$. Therefore we will build these permutations recursively by
adding in a new maximum.

We set up as follows: let $n$ be the length of the permutation; let $k$ be the
number of blocks; let $i$ be the position of the maximum; and let $\ell$ be the
position of the leftmost ceiling point (if there are no ceiling points we set
$\ell = 0$). Let $\hat{a}_{n,k,i,\ell}$ be the number of avoiders where the
maximum is a ceiling point; let $\bar{a}_{n,k,i,\ell}$ be the number of avoiders
where the maximum is a left-to-right minimum; and let $\check{a}_{n,k,i,\ell}$
be the number of avoiders where the maximum is neither a ceiling point nor a
left-to-right minimum. Then we are interested in
\[
  a_{n,k,i,\ell} = \hat{a}_{n,k,i,\ell} + \bar{a}_{n,k,i,\ell} + \check{a}_{n,k,i,\ell}.
\]

First consider adding the maximum as a ceiling point. If we want to add $n$ to
the first block then we must have $m_1 = n-1$ and the leftmost ceiling point can
be anywhere. Therefore,
\[
  \hat{a}_{n,k,1,\ell} = \sum_{m = 0}^k \bar{a}_{n-1,k,1,m}.
\]
Otherwise we want to add $n$ to any of the other blocks. We can do this to a
permutation starting with a maximum as long as it is before the leftmost ceiling
point. If the previous maximum is not a ceiling point then we must add it after
the maximum but before the leftmost ceiling point. We cannot create a new
maximum ceiling point if the previous one is already a ceiling point. Hence, if
$i > 1$,
\[
  \hat{a}_{n,k,i,\ell} = \sum_{m = \ell}^k \bar{a}_{n-1,k,1,m} +
  \sum_{j=1}^{i-1} \sum_{m = i}^k \check{a}_{n-1,k,j,m}.
\]
We can add a maximum to the far left of any length $n-1$ avoider to create a
length $n$ avoider, so we get
\[
  \bar{a}_{n,k,i,\ell} = \sum_{j=1}^{k-1} a_{n-1,k-1,j,\ell -1}.
\]
We can add a new maximum to an existing block so that it is not a ceiling point
as long as it comes before the current maximum, so
\[
  \check{a}_{n,k,i,\ell} = \sum_{j=i}^{k} \hat{a}_{n-1,k,j,\ell} + \check{a}_{n-1,k,j,\ell}.
\]
This together with $\bar{a}_{n,n,1,0} = 1$, $\hat{a}_{n,n-1,i,\ell} = 1$, and
the conditions that $n > k > i$ and $i < \ell$ is enough to enumerate these
permutations recursively.

\begin{proposition} \label{prop:secondrecurrence}
  Let $p = \mpattern{scale=\patttextscale}{3}{1/1,2/2,3/3}{2/0,2/1,2/2,2/3}$ and
  $r = \mpattern{scale=\patttextscale}{3}{1/1,2/2,3/3}{0/1,1/1,2/1,3/1}$
  The number of permutations in $\Av_n(p,r)$ is given by
  \[
    \left\{\sum_{k=0}^n \sum_{j = 0}^k \sum_{\ell = 0}^k
    a_{n,k,i,\ell}\right\}_{\!n\geq 0}
    \!=\; \{1, 1, 2, 5, 14, 43, 143, 509, 1921, 7631, 31725, \ldots\}.
  \]
\end{proposition}

\noindent
This sequence was added to the OEIS by the authors~\cite[\seqnum{A249562}]{motzkin}.

\section{Avoiding \texorpdfstring{$123$}{123}}\label{avoiding123}

It is well known that $|\Av_n(123)|=C_n=\binom{2n}{n}/(n+1)$, the $n$th
Catalan number. Inspired by Sections~\ref{barsection} and
\ref{motzkinsection} we shall derive this fact in a alternative way.

\begin{proposition*}[Hammersley~\cite{hammersley}, Rogers~\cite{rogers}]
  $|\Av_n(123)| = C_n$.
\end{proposition*}

\begin{proof}
  Given a permutation avoiding $123$ we can use its left-to-right minima
  to partition the remaining points into cells.  Each cell must be
  decreasing and the same is true for each row and each column, as noted
  by Claesson and Kitaev~\cite{bijection}. Therefore the permutation is uniquely
  determined by the number of points in each cell.  If a cell is
  non-empty then all the cells strictly above and strictly to the right
  of it will be empty. See e.g.,~Figure~\ref{fig:nonempty} where
  we have five left-to-right minima and are assuming that
  $A \neq \epsilon$.
  \begin{figure}[ht]
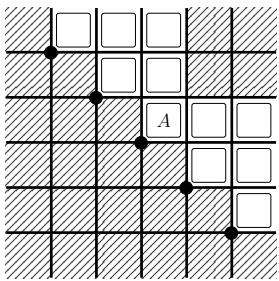

    \[
    \metapatt{0.6}{5/5}{}{1/v,2/v,3/v,4/v,5/v,1/h,2/h,3/h,4/h,5/h}
    {1/5,2/4,3/3,4/2,5/1}
    [0/0,0/1,0/2,0/3,0/4,0/5,1/0,1/1,1/2,1/3,1/4,2/0,2/1,2/2,2/3,3/0,
    3/1,3/2,4/0,4/1,5/0,4/4,4/5,5/4,5/5][1/5/2/6/,2/5/3/6/,3/5/4/6/,
    2/4/3/5/,3/4/4/5/,3/3/4/4/$A$,4/3/5/4/,5/3/6/4/,4/2/5/3/,5/2/6/3/,5/1/6/2/][]
    \]
    \caption{An avoider of $123$ with five left-to-right minima where
      $A \neq \epsilon$}
    \label{fig:nonempty}
  \end{figure}
  This property allows us to construct a larger avoider from two smaller
  ones. See Figure~\ref{fig:av123-sum-1}
  \begin{figure}[ht]
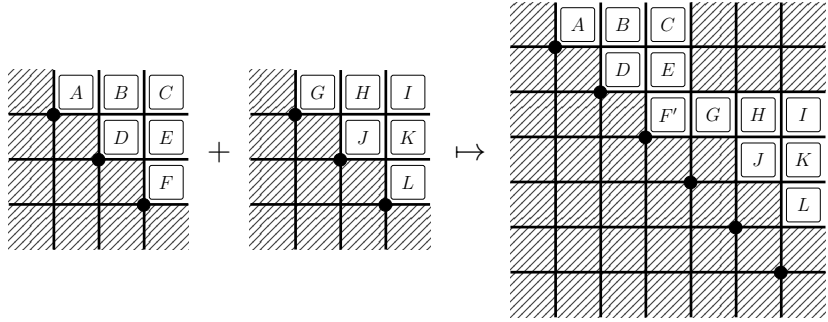

    \[
    \metapatt{0.6}{3/3}{}{1/v,2/v,3/v,1/h,2/h,3/h}{1/3,2/2,3/1}
    [0/0,0/1,0/2,0/3,1/0,1/1,1/2,2/0,2/1,3/0]
    [1/3/2/4/$A$,2/3/3/4/$B$,3/3/4/4/$C$,2/2/3/3/$D$,3/2/4/3/$E$,3/1/4/2/$F$]
    \,+\, \metapatt{0.6}{3/3}{}{1/v,2/v,3/v,1/h,2/h,3/h}{1/3,2/2,3/1}
    [0/0,0/1,0/2,0/3,1/0,1/1,1/2,2/0,2/1,3/0]
    [1/3/2/4/$G$,2/3/3/4/$H$,3/3/4/4/$I$,2/2/3/3/$J$,3/2/4/3/$K$,3/1/4/2/$L$]
    \;\mapsto\;
    \metapatt{0.6}{6/6}{}{1/v,2/v,3/v,4/v,5/v,6/v,1/h,2/h,3/h,4/h,5/h,6/h}
    {1/6,2/5,3/4,4/3,5/2,6/1}
    [0/0,0/1,0/2,0/3,0/4,0/5,0/6,1/0,1/1,1/2,1/3,1/4,1/5,2/0,2/1,2/2,2/3,2/4,3/0,
    3/1,3/2,3/3,4/0,4/1,4/2,5/0,5/1,6/0,6/1,5/2,4/3,5/5,5/6,6/5,6/6,4/5,4/6]
    [1/6/2/7/$A$,2/6/3/7/$B$,3/6/4/7/$C$,2/5/3/6/$D$,3/5/4/6/$E$,3/4/4/5/$F'$,
    4/4/5/5/$G$,5/4/6/5/$H$,6/4/7/5/$I$,5/3/6/4/$J$,6/3/7/4/$K$,6/2/7/3/$L$][]
    \]
    \caption{The sum of two 123-avoiding permutations}
    \label{fig:av123-sum-1}
  \end{figure}
  where $F'$ has one more point than $F$. If we are adding the empty
  permutation, on the left, we instead add a left-to-right minimum. See
  Figure~\ref{fig:av123-sum-2}.
  \begin{figure}[hbtp]
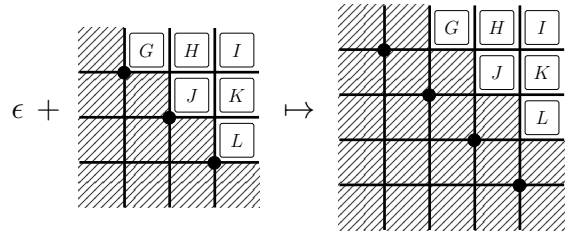

    \[
    \epsilon \,+\, \metapatt{0.6}{3/3}{}{1/v,2/v,3/v,1/h,2/h,3/h}{1/3,2/2,3/1}
    [0/0,0/1,0/2,0/3,1/0,1/1,1/2,2/0,2/1,3/0]
    [1/3/2/4/$G$,2/3/3/4/$H$,3/3/4/4/$I$,2/2/3/3/$J$,3/2/4/3/$K$,3/1/4/2/$L$]
    \;\mapsto\;
    \metapatt{0.6}{4/4}{}{1/v,2/v,3/v,4/v,1/h,2/h,3/h,4/h}{1/4,2/3,3/2,4/1}
    [0/0,0/1,0/2,0/3,0/4,1/0,1/1,1/2,1/3,1/4,2/0,2/1,2/2,2/3,3/0,3/1,3/2,4/0,4/1]
    [2/4/3/5/$G$,3/4/4/5/$H$,4/4/5/5/$I$,3/3/4/4/$J$,4/3/5/4/$K$,4/2/5/3/$L$]
    \]
    \caption{The sum of the empty permutation and a 123-avoiding permutation}
    \label{fig:av123-sum-2}
  \end{figure}
  This construction is reversible. Therefore, if we let $A$ be the generating
  function then it is clear that it will satisfy $A = 1 + x(A-1)A + xA = 1 +
  xA^2$, which is the defining functional equation for the ordinary generating
  function of the Catalan numbers.
\end{proof}

We would like to thank the referee for his detailed comments and suggestions.

\section*{Appendix}

In the tables below the column titled ``Method'' indicates the argument to
confirm the enumeration. In some cases these links are to a proposition or lemma
with the patterns, and in others to a similar case where the same or a similar
argument is used.

\renewcommand{\arraystretch}{1.4}
\begin{longtable}{|c|c|c|c|}
  \hline
  Method & $p$ & $r$ & $|\Av_n(p,r)|$ \\
  \hline
  Prop.~\ref{shadingonce}
  & $(123,\emptyset,\emptyset)$ & $(123,\emptyset,\{1\})$ & $C_n$\\
  & $(132,\emptyset,\emptyset)$ & $(132,\emptyset,\{1\})$ & \\
  & $(132,\emptyset,\emptyset)$ & $(132,\emptyset,\{2\})$ & \\
  \hline
  Prop.~\ref{shadingtwice}
  & $(132,\{1\},\emptyset)$ & $(132,\emptyset,\{1\})$ & $C_n$\\
  & $(132,\{1\},\emptyset)$ & $(132,\emptyset,\{2\})$ & \\
  \hline
  Prop.~\ref{shadingonce}
  & $(123,\emptyset,\emptyset)$ & $(231,\emptyset,\{1\})$ & $\binom{n}{2} + 1$\\
  \hline
  Prop.~\ref{shadeonebox}
  & $(132,\emptyset,\emptyset)$ & $(321,\emptyset,\{1\})$ & $\binom{n}{2} + 1$\\
  \hline
  Lemma~\ref{Shading1} and Prop.~\ref{shadeonebox}
  & $(123,\{2\},\emptyset)$ & $(231,\emptyset,\{1\})$ & $\binom{n}{2} + 1$ \\
  \hline
  Prop.~\ref{shadingonce}
  & $(123,\emptyset,\emptyset)$ & $(132,\emptyset,\{1\})$ & $2^{n-1}$\\
  & $(132,\emptyset,\emptyset)$ & $(213,\emptyset,\{2\})$ & \\
  & $(132,\emptyset,\emptyset)$ & $(231,\emptyset,\{1\})$ & \\
  & $(132,\emptyset,\emptyset)$ & $(312,\emptyset,\{2\})$ & \\
  \hline
  Prop.~\ref{shadingtwice}
  & $(132,\{1\},\emptyset)$ & $(213,\emptyset,\{2\})$ & $2^{n-1}$\\
  & $(132,\{1\},\emptyset)$ & $(231,\emptyset,\{1\})$ & \\
  \hline
  Prop.~\ref{shadeonebox}
  & $(132,\emptyset,\emptyset)$ & $(123,\emptyset,\{1\})$ & $2^{n-1}$\\
  & $(132,\emptyset,\emptyset)$ & $(213,\emptyset,\{1\})$ & \\
  & $(132,\emptyset,\emptyset)$ & $(231,\emptyset,\{2\})$ & \\
  & $(132,\emptyset,\emptyset)$ & $(312,\emptyset,\{1\})$ & \\
  \hline
  Lemma~\ref{Shading1} and Prop.~\ref{shadeonebox}
  & $(123,\{1\},\emptyset)$ & $(132,\emptyset,\{1\})$ & $2^{n-1}$\\
  & $(132,\{1\},\emptyset)$ & $(213,\emptyset,\{1\})$ & \\
  & $(132,\{1\},\emptyset)$ & $(231,\emptyset,\{2\})$ & \\
  & $(132,\{1\},\emptyset)$ & $(312,\emptyset,\{1\})$ & \\
  \hline
  Prop.~\ref{2nfirst}
  & $(123,\{2\},\emptyset)$ & $(312,\emptyset,\{2\})$     & $2^{n-1}$\\
  & $(132,\emptyset,\emptyset)$ & $(321,\emptyset,\{2\})$ & \\
  & $(132,\{2\},\emptyset)$ & $(213,\emptyset,\{1\})$     & \\
  \hline
  Prop.~\ref{2nsecond}
  & $(123,\emptyset,\emptyset)$ & $(231,\emptyset,\{2\})$ & $2^{n-1}$\\
  & $(123,\{1\},\emptyset)$ & $(312,\emptyset,\{1\})$     & \\
  \hline
  Prop.~\ref{choose}
  & $(123,\{2\},\emptyset)$ & $(132,\emptyset,\{2\})$ & \seqnum{A121690}\\
  \hline
  \S~\ref{barsection}
  & $(123,\{1\},\emptyset)$ & $(123,\emptyset,\{1\})$ & \seqnum{A098569}\\
  & $(132,\{2\},\emptyset)$ & $(132,\emptyset,\{2\})$ & \\
  \hline
  Prop. \ref{motzkingen}
  & $(132,\emptyset,\emptyset)$ & $(123,\emptyset,\{2\})$ & $M_n$\\
  \hline
  Lemma~\ref{Shading1} and Prop.~\ref{motzkingen}
  & $(123,\{1\},\emptyset)$ & $(213,\emptyset,\{2\})$ & $M_n$\\
  \hline
  Prop.~\ref{motzkinrecur}
  & $(123,\emptyset,\emptyset)$ & $(132,\emptyset,\{2\})$ & $M_n$\\
  \hline
  Prop.~\ref{lattice}
  & $(132,\{2\},\emptyset)$ & $(231,\emptyset,\{2\})$ & \seqnum{A249563}\\
  \hline
  Prop.~\ref{partition}
  & $(123,\{1\},\emptyset)$ & $(231,\emptyset,\{2\})$ & \seqnum{A249561}\\
  \hline
  \S~\ref{prop:firstrecurrence}
  & $(123,\{1\},\emptyset)$ & $(132,\emptyset,\{2\})$ & \seqnum{A249560}\\
  \hline
  \S~\ref{prop:secondrecurrence}
  & $(123,\{1\},\emptyset)$ & $(123,\emptyset,\{2\})$ & \seqnum{A249562}\\
  \hline
  & $(123,\emptyset,\emptyset)$ & $(321,\emptyset,\{1\})$ & finite\\
  & $(123,\{1\},\emptyset)$ & $(321,\emptyset,\{1\})$ & \\
  & $(123,\{1\},\emptyset)$ & $(321,\emptyset,\{2\})$ & \\
  \hline
  \caption{Enumeration of $\Av_n(p,r)$}\label{big-table}
\end{longtable}

\bibliographystyle{amsalpha}
\bibliography{covincular}

\end{document}

%% file: patternmacros.tex
\newcommand{\shadetheboxes}[1]{
    \foreach \x/\y in {#1}
    \fill[pattern color = black!75, pattern=north east lines] (\x,\y) rectangle +(1,1);
}

\newcommand{\drawthegrid}[1]{
    \draw (0.01,0.01) grid (#1+0.99,#1+0.99);
}

\newcommand{\drawverticallines}[3]{
    \foreach \x in {#2}
    \draw[line width=#3] (\x+0.01,0.01) -- (\x+0.01,#1+0.99);
}

\newcommand{\drawhorizontallines}[3]{
    \foreach \y in {#2}
    \draw[line width=#3] (0.01,\y+0.01) -- (#1+0.99,\y+0.01);
}

\newcommand{\drawtheclpattern}[1]{
    \foreach \x/\y in {#1}
    \filldraw (\x,\y) circle (6pt);
}

\newcommand{\drawclpattern}[2]{
	\foreach[count=\x] \y in {#1}
	{
		\filldraw (\x,\y) circle (#2 pt);
	}
}

\newcommand{\drawspecialbox}[1]{
    \foreach \x/\y/\z/\w/\A in {#1}
    {
        \fill[color = white!100, opacity=1, rounded corners = 1.5pt] (\x+0.125,\y+0.125) rectangle (\z-0.125,\w-0.125);
        \draw[color = black, rounded corners = 1.5pt] (\x+0.125,\y+0.125) rectangle (\z-0.125,\w-0.125);
        \fill[black] (\x/2+\z/2,\y/2+\w/2) node {\A};
    }
}

\newcommand{\drawspecialboxlarge}[1]{
    \foreach \x/\y/\z/\w/\A in {#1}
    {
        \fill[color = white!100, opacity=1, rounded corners = 1.5pt] (\x+0.125,\y+0.125) rectangle (\z-0.125,\w-0.125);
        \draw[color = black, rounded corners = 1.5pt] (\x+0.125,\y+0.125) rectangle (\z-0.125,\w-0.125);
        \fill[black] (\x/2+\z/2,\y/2+\w/2) node {\Large \A};
    }
}

\newcommand{\drawsolidshadedbox}[1]{
    \foreach \x/\y/\z/\w/\A in {#1}
    {
        \fill[color = gray!50, opacity=1, rounded corners=1.5pt] (\x+0.125,\y+0.125) rectangle (\z-0.125,\w-0.125);
        \draw[color = black, rounded corners=1.5pt] (\x+0.125,\y+0.125) rectangle (\z-0.125,\w-0.125);
        \fill[black] (\x/2+\z/2,\y/2+\w/2) node {\A};
    }
}

\newcommand{\drawlabels}[1]{
	\foreach \x/\y/\lab in {#1}
	{
		\draw (\x + 0.5,\y + 0.5) node {\lab};
	}
}

\newcommand{\etcdots}[2]{
	\scalebox{#1}
	{
		\begin{tikzpicture}[baseline=(current bounding box.center)]
			\filldraw (0,2) circle (#2 pt);
			\filldraw (1,1) circle (#2 pt);
			\filldraw (2,0) circle (#2 pt);
		\end{tikzpicture}
	}
}

\newcommand{\etcdotsflipped}[2]{
    \scalebox{#1}
    {
        \begin{tikzpicture}[baseline=(current bounding box.center)]
            \filldraw (0,0) circle (#2 pt);
            \filldraw (1,1) circle (#2 pt);
            \filldraw (2,2) circle (#2 pt);
        \end{tikzpicture}
    }
}

\newcommand{\decr}{\etcdots{0.2}{6}}
\newcommand{\incr}{\etcdotsflipped{0.2}{6}}

\newcommandx{\patt}[9][4={},5={},6={},7={},8={},9=4]
{
	\scalebox{#1}
	{
		\begin{tikzpicture}[baseline=(current bounding box.center)]
			\useasboundingbox (0.0,-.3) rectangle (#2+1,#2+1.3);
			\shadetheboxes{#4}
			\draw (0.01,0.01) grid (#2+1-0.01,#2+1-0.01);

			\drawsolidshadedbox{#6}
			\drawspecialbox{#7}
			\drawspecialboxlarge{#5}
			\drawclpattern{#3}{#9}
			\drawlabels{#8}
		\end{tikzpicture}
	}
}

\newcommandx{\cpatt}[8][4={},5={},6={},7={},8={}]
{
	\scalebox{#1}
	{
		\begin{tikzpicture}[baseline=(current bounding box.center)]
			\useasboundingbox (0.0,-.3) rectangle (#2+1,#2+1.3);
			\shadetheboxes{#4}
			\draw (0.01,0.01) grid (#2+1-0.01,#2+1-0.01);

			\drawsolidshadedbox{#6}
			\drawspecialbox{#7}
			\drawspecialboxlarge{#5}
			\drawclpattern{#3}{4}
			
			\foreach \x/\y in {#8}
			{
				\draw[line width=1] (\x,\y) circle (7 pt);
			}
		\end{tikzpicture}
	}
}

\newcommandx{\metapatt}[8][6={},7={},8={}]
{
    \scalebox{#1}
    {
        \begin{tikzpicture}[baseline=(current bounding box.center)]
					\foreach \width/\height in {#2}
					{
						\useasboundingbox (0.0,-.3) rectangle (\width+1,\height+1.3);
            \shadetheboxes{#6}

            \foreach \pos/\type in {#4}
            {
                \ifthenelse{\equal{\type}{v}}
                {
                    \drawverticallines{\height}{\pos}{1.7pt}
                }
                {
								    \ifthenelse{\equal{\type}{d}}
                    {
                      \draw[densely dashed] (\pos,0) -- (\pos,\height+1);
                    }
										{
											\drawhorizontallines{\width}{\pos}{1.7pt}
										}
                }
            }

            \foreach \pos/\type in {#3}
            {
                \ifthenelse{\equal{\type}{v}}
                {
                    \drawverticallines{\height}{\pos}{0.6pt}
                }
                {
										\drawhorizontallines{\width}{\pos}{0.6pt}
                }
            }

            \drawsolidshadedbox{#8}
            \drawspecialbox{#7}

            \foreach \x/\y/\type in {#5}
            {
                \ifthenelse{\equal{\type}{a}}
                {
                    \draw (\x,\y) circle (6pt);
                    \filldraw (\x,\y) circle (3pt);
                }
                {
                    \filldraw (\x,\y) circle (4pt);
                }
            }
					}
        \end{tikzpicture}
    }
}

\newcommandx{\dpatt}[9][6={},7={},8={},9={}]
{
    \scalebox{#1}
    {
        \begin{tikzpicture}[baseline=(current bounding box.center)]
					\foreach \width/\height in {#2}
					{
						\useasboundingbox (0.0,-.3) rectangle (\width+1,\height+1.3);
            \shadetheboxes{#6}

            \foreach \pos/\type in {#4}
            {
                \ifthenelse{\equal{\type}{v}}
                {
                    \drawverticallines{\height}{\pos}{1.7pt}
                }
                {
								    \ifthenelse{\equal{\type}{d}}
                    {
                      \draw[densely dashed] (\pos,0) -- (\pos,\height+1);
                    }
										{
											\drawhorizontallines{\width}{\pos}{1.7pt}
										}
                }
            }

            \foreach \pos/\type in {#3}
            {
                \ifthenelse{\equal{\type}{v}}
                {
                    \drawverticallines{\height}{\pos}{0.6pt}
                }
                {
										\drawhorizontallines{\width}{\pos}{0.6pt}
                }
            }

            \drawsolidshadedbox{#8}
            \drawspecialbox{#7}

            \foreach \x/\y/\type in {#5}
            {
                \ifthenelse{\equal{\type}{a}}
                {
                    \draw9 (\x,\y) circle (6pt);
                    \filldraw (\x,\y) circle (3pt);
                }
                {
                    \filldraw (\x,\y) circle (4pt);
                }
            }
						
						\drawlabels{#9}
					}
        \end{tikzpicture}
    }
}

\newcommand{\mpattern}[4]{										% mesh pattern for text
  \raisebox{0.6ex}{
  \begin{tikzpicture}[scale=0.35, baseline=(current bounding box.center), #1]
  	\useasboundingbox (0.0,-0.1) rectangle (#2+1.4,#2+1.1);
	
    \shadetheboxes{#4}
    
    \drawthegrid{#2}
    
    \drawtheclpattern{#3}
    
  \end{tikzpicture}}
}